\documentclass[11pt]{imsart}

\usepackage{subfigure}
\usepackage[font=small,labelfont=bf]{caption}
\usepackage{caption}
\usepackage[letterpaper]{geometry}
\usepackage{amsmath, amsthm, amsfonts, amsbsy, thmtools, amssymb}
\declaretheorem{theorem}
\usepackage{thm-restate}
\usepackage{hyperref}
\usepackage{cleveref}
\usepackage[mathscr]{euscript}
\usepackage{tikz}
\usepackage{enumitem}
\usepackage{stmaryrd}
\usepackage{comment}
\usetikzlibrary{arrows}

\makeatletter
\def\namedlabel#1#2{\begingroup
	#2%
	\def\@currentlabel{#2}%
	\phantomsection\label{#1}\endgroup
}
\makeatother

\numberwithin{equation}{section}

\def\PP{\mathbb{P}}

\def\Z{\mathbb{Z}}

\def\R{\mathbb{R}}

\newtheorem{thm}{Theorem}[section]

\newtheorem{cor}[thm]{Corollary}
\newtheorem{conj}[thm]{Question}
\theoremstyle{remark}

\theoremstyle{definition}

\begin{document}

\begin{frontmatter}

\title{Francis Comets' Gumbel last passage percolation}
\runtitle{Comets Gumbel LPP}
\runauthor{Corwin}

\begin{aug}
 \author{\fnms{Ivan}  \snm{Corwin}}
\address{Columbia, Department of Mathematics, 2990 Broadway, NY, NY 10027,  corwin@math.columbia.edu}   



\end{aug}

\begin{abstract}
In 2015, Francis Comets shared with me a clever way to relate a model of directed last passage percolation with i.i.d. Gumbel edge weights to a special case of the log-gamma directed polymer model. To my knowledge, he never wrote this down.
In the wake of his recent passing I am recording Francis' observation along with some associated asymptotics and discussion.
This note is dedicated in memory of Francis whose work in the study of directed polymers defined and refined the field immensely.
\end{abstract}


%

\end{frontmatter}

\section{Main result}

For $E_1,E_2, E_3$ independent exponential random variables of rate $1$, and any $z_1,z_2>0$
\begin{equation}\label{eq:exp}
\min(E_1/z_1,E_2/z_2) \stackrel{(d)}{=} E_3/(z_1+z_2).
\end{equation}
This simple fact will link the log-gamma directed polymer model with parameter $\gamma=1$ (i.e., weights are exponential distributed) to a new model of Gumbel last passage percolation. Such a link between a positive temperature and zero temperature model is  novel and it would be interesting to see if any extensions of this type of connection can be found.

Let us first define the new model of {\it Gumbel last passage percolation}.
Write $X\sim \mathcal{E}$ if $X$ is distributed as an exponential random variable of rate $1$, and $Y\sim \mathcal{G}$ if $Y=\log(1/X)$ for $X\sim \mathcal{E}$. Such a $Y$ is Gumbel distributed so that $\PP(Y\leq y) = e^{-e^{-y}}$. This exponential to Gumbel link is key to matching the models considered here.

Define two families of i.i.d. Gumbel random variables $(U_{m,n})_{n,m\geq 1}$ and $(V_{m,n})_{m,n\geq 1}$. For $m,n\geq 1$ let $T_{m,n}$ be defined via the following recursion relation with boundary conditions. Let $T_{1,1}\sim \mathcal{G}$ and $T_{0,n}=T_{m,0}=-\infty$ for $m,n\geq 2$. For $m,n\geq 1$ except $m=n=1$, define
\begin{equation}\label{trec}
T_{m,n} = \max\left(T_{m-1,n}+U_{m,n}, T_{m,n-1} + V_{m,n}\right).
\end{equation}
This recursion implies that $T_{m,n}$ is a last passage time. Let $\pi:(1,1)\to (m,n)$ denote a lattice path from $(1,1)$ to $(m,n)$ that only takes unit steps to the up and right (i.e., in the direction of the two coordinate axes). Associate to edges between $(m-1,n)$ and $(m,n)$ the weight $U_{m,n}$ and to the edge between $(m,n-1)$ and $(m,n)$ the weight $V_{m,n}$. Associate to an up-right path $\pi$ a weight $Wt(\pi)$ equal to $T_{1,1}$ plus the sum of all edge weights along the path $\pi$. Then it is easily seen that
$$
T_{m,n} = \max_{\pi:(1,1)\to (m,n)} Wt(\pi).
$$
This Gumbel last passage percolation model is illustrated on the right of Figure \ref{fig:LPPs1}

The {\it log-gamma polymer} \cite{timo} can also be defined through a recursion or path formulation. Write $X\sim \Gamma(\gamma)$ if $X$ has the law of a gamma random variable with shape parameter $\gamma$ and scale parameter $1$, and $Y\sim \Gamma^{-1}(\gamma)$ if $Y=1/X$ for some $X\sim \Gamma(\gamma)$, i.e., $Y$ has the law of an inverse gamma random variable.

For $\gamma>0$ fixed and $m,n\geq 1$ let  $w_{m,n}\sim \Gamma^{-1}(\gamma)$ be i.i.d. and let $Z_{1,1}=w_{1,1}$ and for either $m>1$ or $n>1$ let
\begin{equation}\label{zrec}
Z_{m,n} = (Z_{m-1,n}+Z_{m,n-1}) w_{m,n}
\end{equation}
with boundary conditions that $Z_{m,n}=0$ if either $m=0$ and $n>1$ or $m>1$ and $n=0$. This is called the {\it partition function for the log-gamma polymer} from $(1,1)$ to $(m,n)$ and $\log Z_{m,n}$ is called the {\it free energy}. This name is justified by the path formulation for $Z_{m,n}$ achieved by iterating the recursion:
$$Z_{m,n} = \sum_{\pi:(1,1)\to (m,n)} \prod_{(i,j)\in \pi} w_{i,j}$$
where $(i,j)\in \pi$ if and only if the point $(i,j)$ is included in the path $\pi$ (including the starting and ending points). This log-gamma polymer model is illustrated on the left of  Figure \ref{fig:LPPs1}.

\begin{figure}[h]
\scalebox{0.45}{\includegraphics{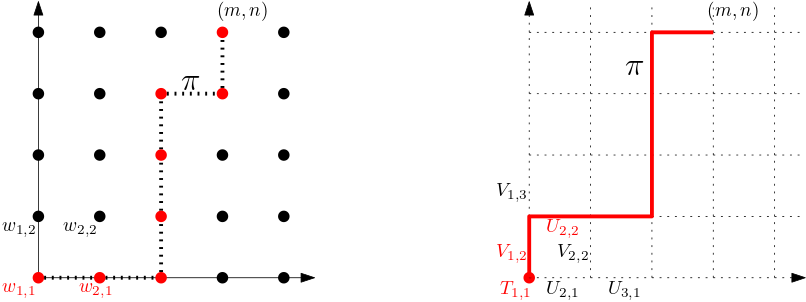}}
\captionsetup{width=\linewidth}
 \caption{Left: The log-gamma polymer partition function is a sum over all paths $\pi$ of the product of inverse-gamma weights $w_{i,j}$ along $\pi$. Right: The Gumbel last passage time is the maximum over all paths $\pi$ of the weights along the edges of $\pi$ ($U$'s for horizontal edges and $V$'s for vertical edges) plus a single weight on the vertex $(1,1)$. All weights are Gumbel distributed.}
\label{fig:LPPs1}
\end{figure}

\begin{theorem}\label{thm:1}
The law of the Gumbel last passage time and the free energy of the log-gamma polymer with $\gamma=1$ are equal in distribution. That is,
\begin{equation}\label{eq:thm}
\Big(T_{m,n}\Big)_{m,n\geq 1} \stackrel{(d)}{=}  \Big(\log Z_{m,n}\Big)_{m,n\geq 1}.
\end{equation}
\end{theorem}
\begin{proof}
For $\gamma=1$, if $X\sim \Gamma(1)$ then it also has exponential rate one law $\mathcal{E}$ and thus, as noted earlier, $\log(1/X)$ has Gumbel law $\mathcal{G}$. On the boundary we can couple the randomness in the Gumbel LPP and log-gamma polymer so that $T_{1,1} = \log w_{1,1}$, $U_{m,1} = w_{m,1}$ for $m\geq 2$ and $V_{1,n} = w_{1,n}$ for $n\geq 2$. It follows immediately that the joint laws of $(T_{m,n})$ and $(\log Z_{m,n})$ match for $m,n\geq 1$ with either $m=1$ or $n=1$ (i.e., the boundary).

The matching of the law in the bulk proceeds differently. We say that $R\subset \Z_{\geq 1}^2$ is  a {\it growth region} if $(1,1)\in R$ and if  all lattice paths $\pi:(1,1)\to (m,n)\in R$ lie in $R$. In other words $R$ looks like a French convention Young diagram with its bottom-left point at $(1,1)$.

Assume that \eqref{eq:thm} has been established for all $(m,n)\in R$. We will prove that the matching consequently also holds after adding a box to the Young diagram of $R$. Inductively this implies \eqref{eq:thm} for all $(m,n)\in \Z_{\geq 1}^2$ as desired. To do this note that owing to the inductively assumed distributional match on $R$ it is possible to couple $T_{m',n'}$ and $\log Z_{m',n'}$ to be equal for all $(m',n')\in R$. Assume now that $(m,n)\notin R$ while $(m-1,n),(m,n-1)\in R$. Under this coupling we claim that the distributions of $T_{m,n}$ and $\log Z_{m,n}$ are equal and hence the region of matching can be expanded to include $(m,n)$ as well. For short-hand, write $z_{1} = Z_{m-1,n}=e^{T_{m-1,n}}$ and $z_{2} = Z_{m,n-1}=e^{T_{m,n-1}}$, and also write
$U_{m,n} = -\log E_1$, $V_{m,n}=-\log E_2$ and $w_{m,n}=1/E_3$ in terms of exponential rate one random variables $E_1,E_2,E_3\sim \mathcal{E}$.
The recursion \eqref{trec} for $T_{m,n}$ and \eqref{zrec} for $Z_{m,n}$ can be rewritten  as
$$e^{T_{m,n}} = \max(z_1/E_1,z_2,E_2),\qquad Z_{m,n} = (z_1+z_2)/E_3.$$
Inverting both sides yields
$$
1/e^{T_{m,n}}=  \min(E_1/z_1,E_2/z_2)\stackrel{(d)}{=} E_3/(z_1+z_2) = 1/Z_{m,n}
$$
where the middle equality follows from \eqref{eq:exp}. This proves the induction.
\end{proof}

The fluctuations of the log-gamma polymer free energy have been probed for general $\gamma$ and shown to be described by the Tracy-Widom GUE distribution as $m$ and $n$ tend to infinity with any fixed ratio (\cite[Theorem 1.2]{BCD} gives the general result, while \cite{BCR} and \cite{KQ} provide earlier asymptotics for some parameters). The following limit theorem for the Gumbel last passage time is a corollary of these asymptotic results along with \Cref{thm:1}. It is  stated for $m=n\to \infty$ though the general $m,n\to \infty$ result is also known.

\begin{cor}\label{cor:1}
Define the digamma function $\Psi(x) = \Gamma'(x)/\Gamma(x)$. Let $C= -2\Psi(1/2)\sim 3.927$ and $\sigma = (-\Psi''(1/2))^{1/3}\sim 2.563$.  Then for all $s\in \R$,
$$
\lim_{n\to \infty} \PP\left(\frac{T_{n,n}-C\cdot n}{\sigma \cdot n^{1/3}}\leq r\right) = F_{\mathrm{GUE}}(r)
$$
where $F_{\mathrm{GUE}}(r)$ is the GUE Tracy-Widom distribution.
\end{cor}
This type of GUE Tracy-Widom distributional limit is only known for a few special models despite being widely believed to be universal. This adds to the list of such models.

\section{Discussion}

The matching in Theorem \ref{eq:thm} was communicated to me by Francis. At the time we were a bit confused what to think of the model since the weights are $\R$-valued which seems a bit unnatural for a last passage percolation model. Since Francis' passing I have reflected a bit further on this model and would like to offer a few observations that show that despite this point, the model is, in fact, quite natural and related to a rather interesting family of $\R_{\geq 0}$-valued last passage percolation and growth models. The key here is that the Gumbel distribution arises through extreme value theory.

Consider any distribution $D$ whose extreme-value statistics are Gumbel distributed, i.e., if $X^{(1)},\ldots, X^{(N)}$ are i.i.d. with distribution $D$ there should exist $C_N$ and $\sigma_N$ such that
\begin{equation}\label{eq:gumbelextreme}
\Big(\max_{i} X^{(i)} - C_N\Big)/\sigma_N\Rightarrow \mathcal{G},
\end{equation}
the Gumbel distribution. Note that the Gumbel has a broad domain of attraction, namely unbounded (in the positive direction) distributions which are not heavy-tailed.

Now imagine an LPP model in which
$T^{(N)}_{1,1}= \max(X^{(i)})$ for $X^{(i)}$ as above and $i\in \{1,\ldots ,N\}$, $T_{0,n}=T_{m,0}=-\infty$ for $m,n\geq 2$ and for $m,n\geq 1$ except $m=n=1$,
$$
T^{(N)}_{m,n} =\max_i\left(T^{(N)}_{m-1,n}+U^{(i)}_{m,n}, T^{(N)}_{m,n-1} + V^{(i)}_{m,n}\right).
$$
Here the maximum is taken over all $i\in \{1,\ldots ,N\}$ over the two types of terms; the $U^{(i)}$ and $V^{(i)}$ are all i.i.d. with distribution $D$.
This $T^{(N)}_{m,n}$ also admits a path maximization formulation in which the maximum is over all paths from $(1,1)$ to $(m,n)$ and over all weights (including the $X^{(1)},\ldots, X^{(N)}$ weights at $(1,1)$)  encountered along those paths. This last passage percolation model, which can be called an $N$ multi-edge LPP, is illustrated in Figure \ref{fig:LPPs2}.

\begin{figure}[h]
\scalebox{0.45}{\includegraphics{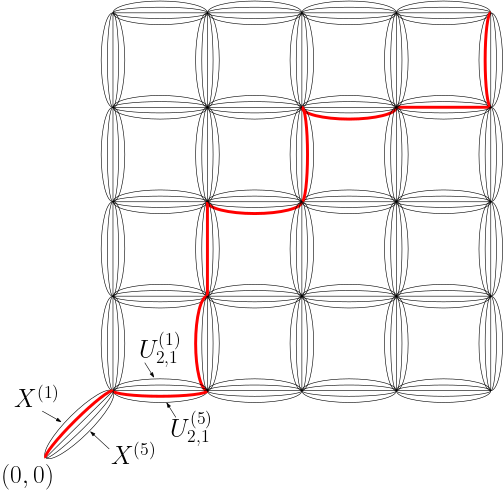}}
\captionsetup{width=\linewidth}
 \caption{The $N=5$ multi-edge last passage percolation model is illustrated here. The last passage time $T^{(N)}_{m,n}$ is equivalent to the maximum over all paths from $(0,0)$ to $(m,n)$ of the sum of the edge weights encountered along the path. Here the red path is supposed to indicate the maximal path. Note that the presence of the multi-edges implies that there are many choices of how to traverse between two vertices, and the maximal choice is taken. In the case of exponential edge weights, the set of all points whose last passage time is at most $t$ defines a Markovian growth model variant of the corner growth model that is described below. Basically, once a vertex is in the growth region, $N$ exponential clocks start running along the outgoing horizontal and vertical edges. Once all of the clocks along the horizontal and vertical edges incoming to a vertex have rung, then the vertex is included in the growth region.}
\label{fig:LPPs2}
\end{figure}

For fixed $m,n$, it follows from the Gumbel extreme-value limit \eqref{eq:gumbelextreme} of $D$ that
$$
\lim_{N\to\infty} \frac{T^{(N)}_{m,n} - C_N(m+n-1)}{\sigma_N} \Rightarrow T_{m,n}.
$$
This and \Cref{cor:1} implies a universality result if one takes $N\to \infty$ and then $m,n\to \infty$.

\begin{conj}
Can the above limits can be taken simultaneously? If $N$ grows fast enough relative to $m$ and $n$  both the centering and scaling of \Cref{cor:1} should hold if $T_{m,n}$ is replaced by $(T^{(N)}_{n,n} - C_N(2n-1))/\sigma_N$. For $N$ growing a bit slower, the law of large numbers centering in \Cref{cor:1} should still hold, but the fluctuations should be affected through a non-universal shift. For $N$ growing too slowly (or not at all), the centering and scaling in \Cref{cor:1} should break down, though by KPZ universality one expects to see $n^{1/3}$ scaling and GUE Tracy-Widom fluctuations with a different centering constant than $C_N$ and scaling constant than $\sigma_N$.
\end{conj}

To prove such results, one presumably should appeal to a rate of convergence result for both the Gumbel distribution from the maximum of $N$ i.i.d. distribution $D$ random variables, as well as the rate of convergence for the log-gamma polymer free energy to the GUE Tracy-Widom distribution. This later result is not present in the literature and would require some work hence the above question is left for future work to a motivated reader.

A particularly simple choice for the distribution $D$ is the exponential rate $1$ distribution $\mathcal{E}$. In this case, $C_N=\log(N)$ and $\sigma_N=1$ in \eqref{eq:gumbelextreme}.
For fixed $N$, this exponential choice of $D$ actually leads to a fairly simple Markovian growth model variant of the corner growth model that is briefly describe here. Consider a state-space given by assignments of numbers in $\{0,\ldots, N\}$ to the site $(1,1)$ as well as every nearest-neighbor edge in $(\Z_{\geq 1})^2$. Write $A(1,1)$ for the assignment at $(1,1)$, $B(m,n)$ for the assignment for the edge from $(m-1,n)$ to $(m,n)$ and $C(m,n)$ for the assignment for the edge from $(m,n-1)$ to $(m,n)$. We introduce a subscript $t$ to denote the Markov process and define boundary conditions so that $B_t(1,n)=N$ for all $t\geq 0$ and $n\geq 2$ and $C_t(m,1)=N$ for all $t\geq 0$ and $m\geq 2$. Initially $A_0(1,1)=B_t(m,n)=C_t(m,n)=0$ for all $(m,n)\in (\Z_{\geq 1})^2$ besides those values of $B_t$ and $C_t$ that have been fixed to equal $N$. These assignments evolve in the following Markovian manner. If $A_t(1,1)=i<N$ then it increases by one at rate $(N-i)^{-1}$. For $(m,n)\in (\Z_{\geq 1})^2$ excluding $(1,1)$, if $B_t(m,n)=C_t(m,n)=N$ then define an auxiliary assignment $A_t(m,n)=N$. For any $(m,n)\in (\Z_{\geq 1})^2$, if $A_t(m,n)=N$ and $B_t(m+1,n)=i<N$ then increase $B_t(m+1,n)$ by one at rate $(N-i)^{-1}$ and likewise if $A_t(m,n)=N$ and $C_t(m,n+1)=i<N$ then increase $C_t(m,n+1)$ by one at rate $(N-i)^{-1}$. In words, once the assignment at $(1,1)$ equals $N$, we start to allow the assignments on edges leaving $(1,1)$ to grow at rates that linearly decrease from $N$ to $1$ as the assignments increase. Once the assignments into a site $(m,n)$ are both $N$, then $(m,n)$ becomes part of the growing substrate and the grow on its outgoing edges is triggered.

To close out this discussion, here is another compelling and rather open-ended question. The log-gamma polymer has a parameter $\gamma$ that was fixed to be $1$ here. Do there exist other models like Gumbel LPP that can be linked to the free energy for $\gamma \neq 1$? And for that matter, are there any other directions in which to generalize this matching, such as to the other solvable polymer models?

\section{Acknowledgements}
This article was mostly written in June 2022 right after Francis' passing while I was delivering lectures at the PIMS summer school. I would like to thank the organizers of that school and acknowledge partial funding from NSF DMS:1952466. My research is partially funded by the NSF through DMS:1811143, DMS:1937254 and DMS:2246576, through the Simons Foundation through a Fellowship in Mathematics grant (\#817655) and an Investigator in Mathematics grant (\#929852), and through the W.~M.~Keck Foundation Science and Engineering Grant on Extreme Diffusion.

\end{document}